\documentclass[a4paper,12pt,leqno]{amsart}
\usepackage[utf8]{inputenc}

\usepackage{amssymb,amsmath}
\usepackage{lmodern}

\usepackage{tikz}
\usetikzlibrary{fadings}
\usetikzlibrary{patterns}
\usetikzlibrary{decorations.pathmorphing}

\newcommand{\bbR}{\mathbb{R}}

\newcommand{\cl}{\operatorname{cl}}

\renewcommand{\epsilon}{\varepsilon}
\renewcommand{\int}{\operatorname{int}}
\renewcommand{\phi}{\varphi}

\newtheorem{thm}{Theorem}
\newtheorem{pro}[thm]{Proposition}
\newtheorem{lem}[thm]{Lemma}

\theoremstyle{definition}
\newtheorem{exm}{Example}

\renewcommand{\geq}{\geqslant}
\renewcommand{\leq}{\leqslant}

\subjclass[2010]{Primary 54C30, 54G20; Secondary 54D70, 26A24}
\keywords{Topology modifications, the Niemytzki plane, real functions}

\title{On some modifications of the Niemytzki plane}
\author{Wojciech Bielas}
\begin{document}

\begin{abstract}
  We present a criterion that compares   modifications of the Niemytzki plane.
  It follows that if usual tangent discs of the Niemytzki plane are replaced by triangles with bounded angles, then the resulting space is not homeomorphic to the former.
\end{abstract}
\maketitle

\section{Introduction}
The Niemytzki plane is a classical object in general topology and its modifications still appear as counterexamples, compare \cite{Engelking} or \cite{Seebach}.
Recently, there have been attempts to generalise it to larger dimensions, \cite{Chatyrko}.
Following Hattori \cite{Hattori}, new topologies on the closed upper plane which lie between the usual topology and the topology of the Niemytzki plane are discussed in \cite{Abuzaid}.
In this paper we  show that  the topology of the Niemytzki plane depends on  the shape of tangent discs and that changing this shape may change the topology.
We present  a quantitative criterion controlling such changes, see Lemma \ref{thm:1}.
Replacing tangent discs with appropriate triangles results in a finer topology, which is not homeomorphic to the Niemytzki plane.
Additionally, we show that obtaining the mentioned criterion is possible because monotone  functions $g\colon \bbR\to\bbR$ are differentiable almost everywhere, which results from the Lebesgue  monotone differentiation theorem, and if the derivative $g'(x)$ exists, then 
\begin{equation}\label{eqn:1}
 \textstyle\underset{h\to 0}{\liminf}\;\frac{g(x+\phi(h))-g(x-\phi(h))}{2\phi(h)} \frac {2\psi(h)}{g(x+\psi(h))-g(x-\psi(h))}\leq 1
 \end{equation}
 for any continuous functions $\phi,\psi\colon (0,\infty)\to(0,\infty)$ such that $\phi\leq\psi$ and $\lim_{h\to 0}\psi(h)=0$.
 If $g'(x)\neq 0$, then the lower limit in (\ref{eqn:1}) equals $1$, but the case $g'(x)=0$ needs an extra effort.

Our notation and topological terminology is fairly standard, as in \cite{Engelking} or \cite{Seebach}, but notions concerning differentiation are taken from standard textbooks on real analysis.

We will use the Jordan curve theorem, compare \cite[p. 31]{Moise}. 

\section{Topologies induced by families of functions}
Here, we present a method for obtaining some modifications of the Niemytzki plane.
If $f\colon [-a,a]\to[0,\infty)$ is a continuous function, then put
$$U(0,f)=\{(0,0)\}\cup\{(x,y)\in\bbR^2\colon f(x)< y<\max f[[-a,a]]\}$$
and  $U(x,f)=\{(y+x,z):(y,z)\in U(0,f)\}$, shortly $U(x,f)=U(0,f)+(x,0)$.
We say that a  family of functions
$$f_n\colon [-a_n,a_n]\to[0,\infty)$$
is \emph{basic}, if  the following conditions are fulfilled.
\begin{itemize}
\item Each $f_n$ is continuous and  even.
\item The range of any $f_n$ is $[0,\frac1n]$, so $\max f_n[[-a_n,a_n]]=\frac1n$.
\item Always the restriction  $f_n|_{[0,a_n]}$ is one-to-one and $f_n(0)=0$.
\item If $m<n$, then the Euclidean closure of $U(0,f_n)$ is contained in the set $U(0,f_m)$.
\end{itemize}
Given a basic family $\{f_n\colon n\geq 1\}$,  on the half-plane $H=\{(x,y)\in\bbR^2\colon y\geq 0\}$ let us introduce a topology generated by  the Euclidean topology and the sets  $U(x,f_n)$.
Note that, if
$$\textstyle f_n(x)=\frac1n-\sqrt{\frac1{n^2}-x^2},$$
then the family $\{f_n|_{[-\frac1n,\frac1n]}\colon n\geq 1\}$ is basic and generates the topology of the Niemytzki plane.

It is straightforward  to check that if functions $f_n\colon [-a_n,a_n]\to[0,\infty)$ constitute a basic family, then the restriction  $f_n|_{[0,a_n]}$ is strictly increasing, the restriction $f_n|_{[-a_n,0]}$ is strictly decreasing, in fact, both restrictions  are homeomorphisms, therefore their images are arcs with endpoints $(0,0)$, $(a_n,\frac1n)$, and $(0,0),(-a_n,\frac1n)$, respectively.
But $U(x,f_n)$ is a translation of $U(0,f_n)$ by the vector $(x,0)$.
The induced topology gives the same  closure of  $U(x,f_n)$ as in the Euclidean topology, so we denote this closure $\cl U(x,f_n)$.

The topology of the Niemytzki plane can be induced by parabolas, which is shown in the  example below.
\begin{exm}\label{exm:1}
  Let $a_n=\frac1{n}$, $p_{n}\colon [-a_n,a_n]\to[0,\frac1n]$, $p_{n}(x)=nx^2$ for  $n\geq 1$.
For all $x\in[-\frac1n,\frac1n]$ we have
  $$\textstyle\frac1n-\sqrt{\frac1{n^2}-x^2}\leq nx^2$$
  and for all $x\in[-\frac1{2n},\frac1{2n}]$ we have
  $$\textstyle nx^2\leq\frac1{2n}-\sqrt{\frac1{4n^2}-x^2}.$$
These give inclusions
$$\textstyle B((0,\frac1{2n}),\frac1{2n})\subseteq U(0,p_n)\subseteq B((0,\frac1n),\frac1n),$$
which are sufficient to check that sets $\{U(0,p_n)\colon n\geq 1\}$ constitute a base at $(0,0)$ for the Niemytzki plane, hence sets $\{U(x,p_n)\colon n\geq 1\}$ constitute a base at $(x,0)$, too.
\end{exm}

Recall that if $0<s<t$, then $\lim_{x\to0}\frac{m|x|^t}{n|x|^s}=0$ for any $m,n\geq 1$.
For this reason,  higher parabolas induce different nevertheless homeomorphic topologies, as is shown in the next example.

\begin{exm}\label{exm:2}
  Let $a_n=n^{-2/s}$ and $p_{s,n}\colon [-a_n,a_n]\to[0,\frac1n]$, and $p_{s,n}(x)=n|x|^s$, where $s\in(0,\infty)$ and $n\geq 1$.
  Let us denote by $N_s$ the half-plane $H$ with the topology induced by the basic family $\{p_{s,n}\colon n\geq 1\}$.
  Then for all $s,t\in(0,\infty)$ the function $f\colon N_s\to N_t$, given by the formula $f(x,y)=(x,y^{t/s})$, is a homeomorphism.
  Each of the spaces $N_s$ is homeomorphic to the Niemytzki plane, since it is homeomorphic to $N_2$.
\end{exm}

Note that for $s=1$ we obtain triangles as neighbourhoods of $(0,0)$, but they are arbitrarily ,,slim'', i.e. their  angles at the vertex $(0,0)$ tend to $0$.

\begin{pro}\label{pro1}
 Given a basic family  $\{f_n\colon n\geq 1\}$, if $a<b$ and $U(a,f_n)\cap U(b,f_n)\neq\emptyset$, then
  \begin{enumerate}
  \item[$(a)$] the complement of $U(a,f_n)\cup U(b,f_n)$  has two connected components $C$ and $D$, one of which, say $C$, is bounded,
  \item[$(b)$] for any $(u,w)\in C$ and $(v,z)\in \cl U(a,f_n)\cap \cl U(b,f_n)$ we have
    $$\textstyle w\leq f_n(\frac{b-a}2)\leq z,$$
      \item[$(c)$] for any  $(u,w)\in C$ there exists $c,d\in[a,b]$ such that $c\leq d$ and $(u,w)= (\frac{c+d}2,f_n(\frac{d-c}2)).$
  \end{enumerate}
\end{pro}

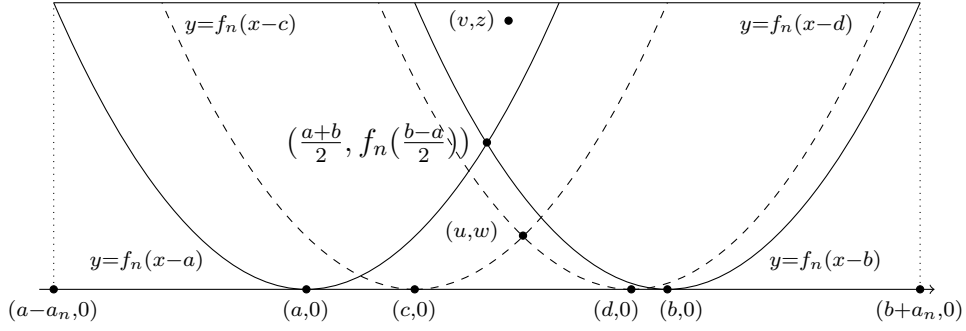
\begin{center}
  \begin{tikzpicture}[scale=.95]
    \draw[->] (-6.2,0) -- (6.2,0);
    \def\x{-2.5}
    \def\y{2.5}
    \def\r{.05}
    \draw[fill] (\x,0) circle[radius=\r];
    \draw[fill] (\x,-.3) node {$_{(a,0)}$};
    \draw[fill] (\y,0) circle[radius=\r];
    \draw[fill] (\y+.2,-.3) node {$_{(b,0)}$};
    \draw (\x,0) parabola (\x+3.5,4);
    \draw (\x,0) parabola (\x-3.5,4);
    \draw (\y,0) parabola (\y-3.5,4);
    \draw (\y,0) parabola (\y+3.5,4);
    \draw (-6,4) -- (6,4);
    \draw[fill] (0,2.05) circle[radius=\r];
    \draw[fill] (-1.5,2.05) node {$(\frac{a+b}2,f_n(\frac{b-a}2))$};
    \draw (-4.7,.4) node {$_{y=f_n(x-a)}$};
    \draw (4.7,.4) node {$_{y=f_n(x-b)}$};
    \draw[fill] (.5,.75) circle[radius=\r];
    \draw[fill] (-.2,.8) node {$_{(u,w)}$};
    \def\z{-1}
    \def\w{2}
    \draw[dashed] (\z,0) parabola (\z+3.5,4);
    \draw[dashed] (\z,0) parabola (\z-3.5,4);
    \draw[dashed] (\w,0) parabola (\w-3.5,4);
    \draw[dashed] (\w,0) parabola (\w+3.5,4);
    \draw[fill] (2,0) circle[radius=\r];
    \draw[fill] (-1,-.3) node {$_{(c,0)}$};
    \draw[fill] (-1,0) circle[radius=\r];
    \draw[fill] (1.8,-.3) node {$_{(d,0)}$};
    \draw (-3.4,3.7) node {$_{y=f_n(x-c)}$};
    \draw (4.3,3.7) node {$_{y=f_n(x-d)}$};
    \draw[fill] (.3,3.75) circle[radius=\r];
    \draw[fill] (-.2,3.75) node {$_{(v,z)}$};
    \draw[fill] (\x-3.5,0) circle[radius=\r];
    \draw[fill] (\x-3.5,-.3) node {$_{(a-a_n,0)}$};
    \draw[dotted] (\x-3.5,0) -- (\x-3.5,4);
    \draw[fill] (\y+3.5,0) circle[radius=\r];
    \draw[fill] (\y+3.5,-.3) node {$_{(b+a_n,0)}$};
    \draw[dotted] (\y+3.5,0) -- (\y+3.5,4);

  \end{tikzpicture}
\end{center}

\begin{proof}
  Condition $(a)$ follows directly from the Jordan curve theorem, where we have to consider two plane simple closed curves, determined by arc with endpoints $(a,0),(\frac{a+b}2,f_n(\frac{b-a}2)),(b,0)$ (the boundary of $C$) and $(a,0),(b,0),(b+a_n,\frac1n),(a-a_n,\frac1n)$ (the boundary of $D$).
Condition $(b)$ can be verified by examining the above picture.

Given $(u,w)\in C$, consider translations of the graph of $f_n$ so that
$$c=u-f_n^{-1}(w)\mbox{ and }d=u+f^{-1}_n(w).$$

If $u=\frac{c+d}2\geq\frac{a+b}2$, then
$$\textstyle c=u-f^{-1}_n(w)\geq \frac{a+b}2-\frac{b-a}2=a,$$
by condition $(b)$.
Since $c\leq u\leq b$, we have $c\leq b$.

If $u=\frac{c+d}2\leq\frac{a+b}2$, then
$$\textstyle u-a\leq \frac{a+b}2-a=\frac{b-a}2,$$
hence $u-a$ belongs to the domain of $f_n$ and we have $f_n(u-c)\leq f_n(u-a)$. Since $u-c,u-a\geq 0$, we obtain $u-c\leq u-a$, which gives $a\leq c$.
Similarly, we show that $d\in[a,b]$, which  completes the proof of condition $(c)$.
\end{proof}

\section{Topologies induced by basic families}
Now, we present a necessary condition (a criterion) to have a homeomorphism between two topologies induced by basic families.

\begin{lem}\label{thm:1}
  Assume that $\{p_n\colon n\geq 1\}$ and $\{t_n\colon n\geq 1\}$ are basic families, inducing topologies $\tau_0$ and $\tau_1$, respectively.
  If spaces $(H,\tau_0)$ and $(H,\tau_1)$ are homeomorphic, then there exist $\epsilon>0$ and functions $\gamma,\delta\colon (0,\epsilon)\to(0,\infty)$ such that for every  $m>n$ there exists  $k$ such that $\lim_{h\to 0}\delta(h)=0$ and $\liminf_{h\to0}\gamma(h)\leq 1$, and
  $$t_1(\delta(x))\leq t_k\bigg(\frac{p_m^{-1}(x)\delta(x)\gamma(x)}{p^{-1}_n(x)}\bigg)$$
  for any $x\in (0,\epsilon)$.
\end{lem}

The proof of this criterion shall be preceded by some lemmas, hence  it is  postponed to the last section.

First, we will show how to use the lemma, showing that the topology generated by triangles with bounded angles is not homeomorphic to the Niemytzki plane.

\begin{exm}\label{exm:3}
  Let $\alpha\in(0,\frac\pi2)$, $a_n=(n\tan\frac{\alpha n}{n+1})^{-1}$,
$$\textstyle t_{\alpha,n}\colon [-a_n,a_n]\to[0,\frac1n]\mbox{ and }t_{\alpha,n}(z)=|z|\tan\frac{\alpha n}{n+1}.$$
  It is the basic family which induces a topology not homeomorphic to the Niemytzki plane.
  Indeed, suppose that there exists such a homeomorphism onto the Niemytzki plane.
  By Lemma \ref{thm:1}, there exist $\epsilon>0$ and functions $\gamma,\delta\colon (0,\epsilon)\to(0,\infty)$ satisfying the hypothesis of the theorem for basic families $\{t_{\alpha,n}\colon n\geq 1\},\{p_n\colon n\geq1\}$, where functions $p_n$ are given by the formula $p_n(x)=nx^2$.
For $x\in(0,\epsilon)$ we have $p^{-1}_m(x)=\sqrt{x/m}$, so substituting $z=\frac{p_m^{-1}(x)}{p^{-1}_n(x)}\delta(x)\gamma(x)$, in the counter and then $z=\delta(x)$ in the denominator, we get
$$
1\leq\frac{t_{\alpha,k}\big(\frac{p_m^{-1}(x)}{p^{-1}_n(x)}\delta(x)\gamma(x)\big)}{t_{\alpha,1}(\delta(x))}=\frac{\frac{\sqrt{x/m}}{\sqrt{x/n}}\delta(x)\gamma(x)\tan\frac{\alpha k}{k+1}}{\delta(x)\tan(\alpha/2)}=$$
$$    \frac{\sqrt{\frac nm}\gamma(x)\tan\frac{\alpha k}{k+1}}{\tan(\alpha/2)}.$$

Fix $k,n$ and choose $m>n$ such that
$$    \frac{\sqrt{\frac nm}\tan\alpha}{\tan(\alpha/2)}<1.$$
So, there exists $x\in(0,\epsilon)$ such that
$$\frac{\sqrt{\frac nm}\gamma(x)\tan\frac{\alpha k}{k+1}}{\tan(\alpha/2)}<\gamma(x)
\frac{\sqrt{\frac nm}\tan\alpha}{\tan(\alpha/2)}<1,$$
since $\liminf_{x\to 0^+}\gamma(x)\leq 1$; a contradiction.
\end{exm}

Note that if $\tau_1$ is the topology of the Niemytzki plane, $\tau_2$---the topology induced by  $\{t_{\alpha,n}\colon n\geq1\}$, and $\tau_3$ is the topology induced by $\{p_{1,n}\colon n\geq1\}$, then $\tau_1\subsetneq\tau_2\subsetneq \tau_3$.
In Example \ref{exm:2}, substituting $s=1$ and $s=2$, we get $(H,\tau_1)\cong (H,\tau_3)$.
But by Example \ref{exm:3}, we have $(H,\tau_1)\not\cong(H,\tau_2)$.

\begin{exm}
  Consider a basic family $w_n(x)=|x|^{(n+1)/n}$ and  suppose that the family $\{w_n\colon n\geq 1\}$ induces the topology homeomorphic to the Niemytzki plane.
Take  $p_n(x)=nx^2$ and let functions $\delta,\gamma$ be as in Lemma \ref{thm:1}.
 Then $$1\leq\frac{p_k\big(\frac{w_m^{-1}(x)}{w_n^{-1}(x)}\delta(x)\gamma(x)\big)}{p_1(\delta(x))}=\frac{k\big(|x|^{\frac m{m+1}-\frac n{n+1}}\delta(x)\gamma(x)\big)^2}{\delta(x)^2}=kx^{2(\frac m{m+1}-\frac n{n+1})}\gamma(x)^2.$$
 For any $m>n$ and $k$ we have
 $$\liminf_{x\to 0^+}kx^{2(\frac m{m+1}-\frac n{n+1})}\gamma(x)^2=0;$$
 this contradicts Lemma \ref{thm:1}.
 Thus, topologies induced by the basic families $\{p_n\colon n\geq 1\}$ and $\{w_n\colon n\geq 1\}$ are not homeomorphic.
 
 The topology induced by $\{w_n\colon n\geq 1\}$ is not homeomorphic to the topology induced by $\{t_{\alpha,n}\colon n\geq 1\}$ , where $\alpha\in(0,\frac\pi2)$ and functions $t_{\alpha,n}$ are as in Example \ref{exm:3}.
 We have
\begin{multline*}  \frac{t_{\alpha,k}\big(\frac{w_m^{-1}(x)}{w_n^{-1}(x)}\delta(x)\gamma(x)\big)}{t_{\alpha,1}(\delta(x))}=\frac{x^{\frac m{m+1}-\frac n{n+1}}\delta(x)\gamma(x)\tan((1-\frac1{k+1})\alpha)}{\delta(x)\tan(\alpha/2)}=\\
  x^{\frac m{m+1}-\frac n{n+1}}\gamma(x)\frac{\tan((1-\frac1{k+1})\alpha)}{\tan(\alpha/2)}<x^{\frac m{m+1}-\frac n{n+1}}\gamma(x)\frac{\tan\alpha}{\tan(\alpha/2)}.
\end{multline*}
If $m>n$, then
$$\liminf_{x\to 0^+}x^{\frac m{m+1}-\frac n{n+1}}\gamma(x)\frac{\tan\alpha}{\tan(\alpha/2)}=0,$$
contrary to Lemma \ref{thm:1}.
  \end{exm}

  If $\lim_{x\to 0}\phi(x)=\lim_{x\to 0}\psi(x)=0$, then the limit  $$\liminf_{x\to 0}\frac{\phi(x)}{\psi(x)}$$ may be arbitrarily large.
The following lemma shows a useful case, where such a limit has to be bounded.
  
\begin{lem}\label{lem:1}
  Let  $h,\phi,\psi\colon(0,\delta)\to(0,\infty)$ be functions such that  $\phi\leq\psi$ and $\lim_{x\to 0}h(x)=\lim_{x\to 0}\psi(x)=0$.
  If $\phi,\psi$ are continuous, then
  $$\liminf_{x\to 0}\frac{h(\phi(x))}{h(\psi(x))}\leq 1.$$
\end{lem}

\begin{proof}
  Consider the case where  $0$ is a limit point of the set
  $$A=\{x>0\colon \phi(x)=\psi(x)\}.$$
If $(x_n)\subseteq A$ is a sequence convergent to $0$, then
  $$\lim_{n\to\infty}\frac{h(\phi(x_n))}{h(\psi(x_n))}=1.$$
Now, assume that $0$ is not a limit point of  $A$, so there exists  $\eta>0$ such that if $x\in(0,\eta)$, then $\phi(x)<\psi(x)$.
  Suppose that
  $$\liminf_{x\to 0}\frac{h(\phi(x))}{h(\psi(x))}>1.$$
Without loss of generality,  assume that  $\frac{h(\phi(x))}{h(\psi(x))}\geq 1$ for each $x\in(0,\eta)$.
Inductively  define a sequence $(x_k)$, putting  $x_0=\eta/2$.
  If the term $x_k$ is defined, then   $\lim_{x\to0}\psi(x)=0<\phi(x_k)<\psi(x_k)$, hence there exists  $x_{k+1}\in(0,x_k)$ such that $\psi(x_{k+1})=\phi(x_k)$, since $\psi$ is continuous.
By the definition, the sequence $(x_k)$ is decreasing, hence convergent to   $a\in[0,\eta/2]$.
 Since $\phi$ is also continuous, then 
  $$\psi(a)=\lim_{k\to\infty}\psi(x_{k+1})=\lim_{k\to\infty}\phi(x_k)=\phi(a),$$
which implies $a=0$ and $\lim_{k\to\infty}\psi(x_k)=0$.
  Moreover
  $$h(\psi(x_{k+1}))=h(\phi(x_k))\geq h(\psi(x_k)),$$
  which implies that
  $$0=\lim_{x\to 0}h(\psi(x_k))\geq h(\psi(x_0))>0;$$
  a contradiction.
\end{proof}

Below we will use a real number $x$ and the point $(x,0)$ interchangeably.
Fix basic families $\{p_n\colon n\geq 1\}$ and $\{t_n\colon n\geq 1\}$, which induce the topologies $\tau_0$ and $\tau_1$, respectively.
For shorter symbols, we will use $P=(H,\tau_0,)$ and $T=(H,\tau_1)$.
Let
$$f=(f_1,f_2)\colon P\to T$$
 be a homeomorphism.
 Note that sets
 $$(\cl U(0,p_n))\setminus U(0,p_m)\mbox{ and }(\cl U(0,t_n))\setminus U(0,t_m)$$
 are not compact for any $m>n$, hence $f_2(x,0)=0$ for any $x\in\bbR$ and we define $g\colon\bbR\to\bbR$, $g(x)=f_1(x,0)$.

\begin{lem}
  There exists an interval $I\subseteq\bbR$ such that $g[I]$ is the interval with endpoints $g(\inf I),g(\sup I)$ and the restriction $g|_{I}$ is monotone.
\end{lem}

\begin{proof}
  Fix $a<c$ and assume that $g(a)<g(c)$.
  If $g[(a,c)]=(g(a),g(c))$, then it remains to prove  the monotonicity.
  Otherwise, we can assume that there exists $b\in (a,c)$ such that $g(b)<g(a)$.
  Let $s_0,s_1,s_2$ be vertical segments and $s_3,s_4$ be horizontal segments and let $s_5$ be a broken line, connecting endpoints of  $s_0$ and $s_2$, arranged as in the picture on the left.
  
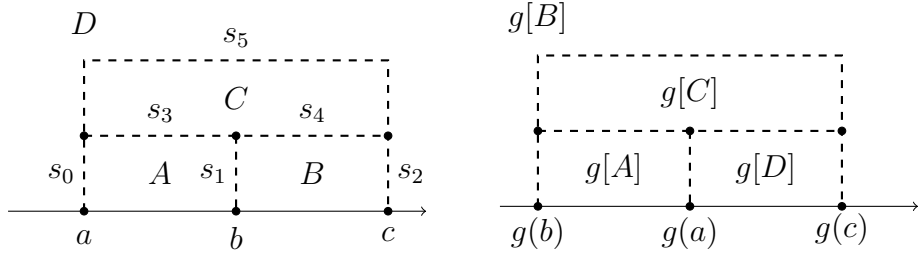
\begin{center}
    \begin{tikzpicture}
      \draw[->](-3,0)--(2.5,0);
      \draw[fill] (-2,0) circle[radius=.05];
            \draw[fill] (0,0) circle[radius=.05];
      \draw[fill] (2,0) circle[radius=.05];
      \draw[fill] (-2,1) circle[radius=.05];
      \draw[fill] (0,1) circle[radius=.05];
      \draw[fill] (2,1) circle[radius=.05];

      \draw (-2,-.35) node {$a$};
            \draw (0,-.35) node {$b$};
      \draw (2,-.3) node {$c$};
      \draw (-2,2.5) node {$D$};
      \draw (0,1.5) node {$C$};
      \draw (1,.5) node {$B$};
      \draw (-1,.5) node {$A$};
      \draw[dashed,thick] (-2,0)--(-2,2)--(2,2)--(2,0);
      \draw[dashed,thick] (-2,1)--(2,1);
      \draw[dashed,thick] (0,1)--(0,0);
      \draw (-2.3,.5) node {$s_0$};
      \draw (-.3,.5) node {$s_1$};
      \draw (2.3,.5) node {$s_2$};
      \draw (-1,1.3) node {$s_3$};
      \draw (1,1.3) node {$s_4$};
      \draw (0,2.3) node {$s_5$};
    \end{tikzpicture}\quad\quad
        \begin{tikzpicture}
      \draw[->](-2.5,0)--(3,0);
      \draw[fill] (-2,0) circle[radius=.05];
            \draw[fill] (0,0) circle[radius=.05];
      \draw[fill] (2,0) circle[radius=.05];
      \draw[fill] (-2,1) circle[radius=.05];
      \draw[fill] (0,1) circle[radius=.05];
      \draw[fill] (2,1) circle[radius=.05];

      \draw (-2,-.35) node {$g(b)$};
            \draw (0,-.35) node {$g(a)$};
      \draw (2,-.3) node {$g(c)$};
      \draw (-2,2.5) node {$g[B]$};
      \draw (0,1.5) node {$g[C]$};
      \draw (1,.5) node {$g[D]$};
      \draw (-1,.5) node {$g[A]$};
      \draw[dashed,thick] (-2,0)--(-2,2)--(2,2)--(2,0);
      \draw[dashed,thick] (-2,1)--(2,1);
      \draw[dashed,thick] (0,1)--(0,0);
    \end{tikzpicture}
  \end{center}
Then $H\setminus (s_0\cup \ldots\cup s_5)=A\cup B\cup C\cup D$ is a union of open, with respect to the topology induced by $\{p_n\colon n\geq 1\}$, and disjoint subsets.
  Any point $x\in (a,b)$ can be connected with $a$ and $b$ by arcs disjoint from $B\cup C\cup D$.
  Thus $g(x)$ must be a point which can be connected with $g(a)$ and $g(b)$ by arcs disjoint from $g[B]\cup g[C]\cup g[D]$, hence  $g(x)\in(g(b),g(a))$.
  Arguing analogously, we get $g^{-1}(y)\in (a,b)$ for any $y\in (g(b),g(a))$.
  Thus $g[(a,b)]=(g(b),g(a))$.
  If $g(b)>g(c)$, then, arguing analogously, we get $g[(b,c)]=(g(c),g(b))$.

  Fix $x,y\in (a,b)$, $x<y$ and suppose that $g(x)<g(y)$.
  Then we have disjoint broken lines: $l$ connecting $a$ with $x$, and $q$ connecting $y$ with $b$, as in the following picture.
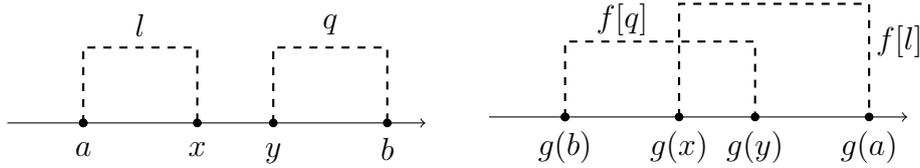
\begin{center}
    \begin{tikzpicture}
      \draw[->](-3,0)--(2.5,0);
      \draw[fill] (-2,0) circle[radius=.05];
      \draw[fill] (2,0) circle[radius=.05];
      \draw[fill] (-2,0) circle[radius=.05];
      \draw[fill] (-.5,0) circle[radius=.05];
      \draw[fill] (.5,0) circle[radius=.05];

\draw (-1.25,1.3) node {$l$};
\draw (1.25,1.3) node {$q$};
      \draw (-2,-.35) node {$a$};
      \draw (-.5,-.35) node {$x$};
                  \draw (.5,-.4) node {$y$};

      \draw (2,-.35) node {$b$};
      \draw[dashed,thick] (-2,0)--(-2,1)--(-.5,1)--(-.5,0);
      \draw[dashed,thick] (.5,0) -- (.5,1)--(2,1)--(2,0);
    \end{tikzpicture}\quad\quad
    \begin{tikzpicture}
      \draw (-1.25,1.3) node {$f[q]$};
\draw (2.4,1) node {$f[l]$};

      \draw[->](-3,0)--(2.5,0);
      \draw[fill] (-2,0) circle[radius=.05];
      \draw[fill] (2,0) circle[radius=.05];
      \draw[fill] (-2,0) circle[radius=.05];
      \draw[fill] (-.5,0) circle[radius=.05];
      \draw[fill] (.5,0) circle[radius=.05];

      \draw (-2,-.4) node {$g(b)$};
      \draw (-.5,-.4) node {$g(x)$};
                  \draw (.5,-.4) node {$g(y)$};

      \draw (2,-.4) node {$g(a)$};
      \draw[dashed,thick] (-2,0)--(-2,1)--(.5,1)--(.5,0);
      \draw[dashed,thick] (-.5,0) -- (-.5,1.5)--(2,1.5)--(2,0);
    \end{tikzpicture}\quad\quad
  \end{center}
  If $g(x)<g(y)$, then arcs $f[l],f[q]$ intersect; a contradiction.
  In the remaining cases we argue similarly.
\end{proof}

\noindent\emph{Proof of Lemma \ref{thm:1}.}
By the previous lemma, we can assume that the restriction $g|_{(a,b)}$ is an increasing bijection onto $(g(a),g(b))$.
Continuity of $f$ implies that for each $x\in\bbR$ there exists  $n_x$ such that 
  $$\textstyle f[U(x,p_{n_x})]\subseteq U(g(x),t_1).$$
  By the Baire Category Theorem,  there exists $n$ and a non-empty interval $(c,d)\subseteq(a,b)$ such that the set $\{x\in(c,d)\colon n_x=n\}$ is dense in $(c,d)$.
  We can assume $(c,d)=(a,b)$.
  We can also assume that neighbourhoods   $U(a,p_n),U(b,p_n)$ intersect.

  Fix $c,d\in (a,b)$ such that $c<d$ and $n_c=n_d=n$.
  Let
  \begin{equation}
    \label{eq:1}
    \textstyle
    \begin{cases}
    u(c,d)=\frac{c+d}2,\\
    w(c,d)=p_n(\frac{d-c}2).  
    \end{cases}
  \end{equation}
The point $(u(c,d),w(c,d))$ belongs to $\cl_{N_P} U(c,p_n)\cap \cl_{N_P}U(d,p_n)$, hence
 $$f(u(c,d),w(c,d))\in \cl_{N_T}U(g(c),t_1)\cap \cl_{N_T}U(g(d),t_1).$$
 The point $\textstyle (\frac{g(c)+g(d)}2,t_1(\frac{g(d)-g(c)}2))$
has the least ordinate among all points of the intersection $\cl_{N_T}U(g(c),t_1)\cap \cl_{N_T}U(g(d),t_1)$ (Proposition \ref{pro1} $(b)$ for basic family $\{t_n\colon n\geq 1\}$), hence
  \begin{equation}
    \label{eq:2}
    \textstyle t_1(\frac{g(d)-g(c)}2)\leq f_2(u(c,d),w(c,d)).
  \end{equation}
  Functions $g,t_1,f_2,u,w$ are continuous, which implies that (\ref{eq:2}) holds for any $c,d\in(a,b)$, $c<d$.
  If $(u,w)$ is a point of the bounded connected component of $H\setminus(U(a),p_n)\cap U(b,p_n))$, then there exists $c,d\in (a,b)$ such that $c<d$, $u=\frac{c+d}2$ and $p_n(\frac{d-c}2)=w$ (Proposition \ref{pro1} $(c)$ for the family $\{p_n\colon n\geq 1\}$).
  Then  $c=u-p_n^{-1}(w)$ and $d=u+p_n^{-1}(w)$.
  Now we can rewrite (\ref{eq:2}) in terms of $(u,w)$:
  \begin{equation}
    \label{eq:3}
 t_1\bigg(\frac{g(u+p_n^{-1}(w))-g(u-p_n^{-1}(w))}2\bigg)\leq f_2(u,w).
\end{equation}

Fix $m>n$.
For each $u\in(a,b)$ there exists $k_u$ such that 
  $$U(g(u),t_{k_u})\subseteq f[U(u,p_m)].$$
By the Baire Category Theorem, there exists a non-empty interval $(c,d)\subseteq (a,b)$ and $k$ such that the set $\{u\in (c,d)\colon k_u=k\}$ is dense in $(c,d)$.
  We can assume that $(c,d)=(a,b)$ and the intersection $U(g(a),t_k)\cap U(g(b),t_k)$ is non-empty.
  
Fix $a<c<d<b$.
Then the intersection $f[U(c,p_m)]\cap f[U(d,p_m)]$ is non-empty, hence $U(c,p_m)\cap U(d,p_m)\neq\emptyset$.
Let $u=\frac{c+d}2$ and $w=p_m(\frac{d-c}2)$.  
      \def\t{-.4}
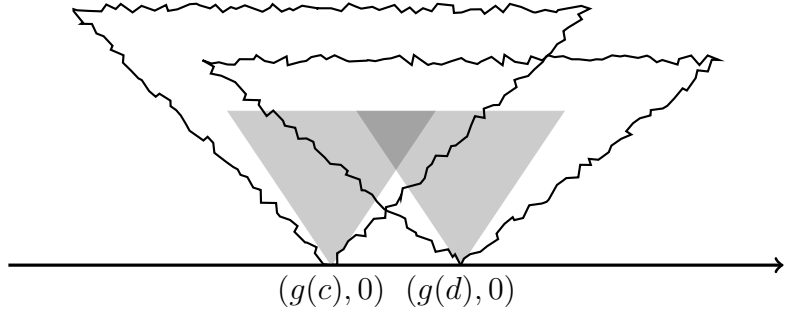
      \begin{center}
        \begin{tikzpicture}[scale=1.7]
          \draw[very thick,->] (-3.5,0) -- (2.5,0);
          \draw[thick,decorate, decoration={random steps,segment length=3pt,amplitude=2pt}] (-1,0) -- (-3,2) -- (1,2) -- (-1,0);
          \draw[thick,decorate, decoration={random steps,segment length=3pt,amplitude=2pt}] (-1+1,0) -- (-3+1,2+\t) -- (1+1,2+\t) -- (-1+1,0);      
          \draw[fill,opacity=.2] (-1,0) -- (-3-3*\t,2+2*\t) -- (1+3*\t,2+2*\t) -- (-1,0);
          \draw[fill,opacity=.2] (-1+1,0) -- (-3+1-3*\t,2+2*\t) -- (1+1+3*\t,2+2*\t) -- (-1+1,0);
          \draw (-1,-.2) node {$(g(c),0)$};
          \draw (0,-.2) node {$(g(d),0)$};
        \end{tikzpicture}
      \end{center}
Now, if  $l$ is a segment connecting  $(u,w)$ with $(u,0)$, then it is disjoint from  $U(c,p_m)\cup U(d,p_m)$.
     Its image $f[l]$ is an arc, connecting $f(u,w)$ with $f(u,0)=(g(u),0)$, disjoint from  $f[U(c,p_m)\cup U(d,p_m)]$.
     This implies that the point $f(u,w)$ is in the bounded connected component of  $H\setminus f[U(c,p_m)\cup U(d,p_m)]$, hence it is in the bounded connected component $E$ of  $H\setminus (U(g(c),t_k)\cup U(g(d),t_k))$.
Thus
\begin{equation}
  \label{eq:5}
  \textstyle f_2(u,w)\leq  t_k(\frac{g(d)-g(c)}2),
\end{equation}
because $t_k(\frac{g(d)-g(c)}2)$ is the largest ordinate of all points of $E$ (Proposition \ref{pro1} $(b)$).
      Functions $f_2,g,t_k$  and  $p_m$ are continuous, which implies that (\ref{eq:5}) holds for any  $(u,w)\in E$, where $c=u-p_m^{-1}(w)$ and $d=u+p^{-1}_m(w)$.
      Substituting the last two equations to (\ref{eq:5}) we obtain 
      \begin{equation}
        \label{eq:4}
        f_2(u,w)\leq t_k\bigg(\frac{g(u+p^{-1}_m(w))-g(u-p^{-1}_m(w))}2\bigg)
      \end{equation}
      for any $(u,w)\in E$.
Let us define the following derivative quotient:
$$I(u,r,w)=\frac{g(u+r(w))-g(u-r(w))}{2r(w)},$$
for any function $r\colon(0,\infty)\to(0,\infty)$.
       From (\ref{eq:3}) and (\ref{eq:4}) we get
      $$\textstyle t_1(I(u,p_n^{-1},w)p_n^{-1}(w))\leq t_k(I(u,p_m^{-1},w)p_m^{-1}(w)).$$
Let $\delta(w)=I(u,p_n^{-1},w)p_n^{-1}(w)$ and $\gamma(w)=\frac{I(u,p_m^{-1},w)}{I(u,p_n^{-1},w)}$.
      Then
      $$ t_1(\delta(w))\leq t_k\bigg(\frac{p_m^{-1}(w)}{p_n^{-1}(w)}\delta(w)\gamma(w)\bigg),$$
     which proves the second inequality of the hypothesis of Lemma \ref{thm:1}.

There exists a point $u\in (a,b)$ such that the function $g$ is differentiable at $u$.
          If $g'(u)\neq 0$, then $\lim_{w\to 0^+}\gamma(w)=1$.
          Assume that $g'(u)=0$.
Then the function
$$\textstyle [0,\frac1m]\ni w\mapsto h(w)=
\begin{cases}
  0,&\mbox{if }w=0,\\
  \frac{g(u+w)-g(u-w)}{2w},&\mbox{if }w>0,
\end{cases}$$
is continuous.
Thus functions $h,p_m^{-1},p_n^{-1}$ (restricted to the interval $(0,\frac1m)$) satisfy assumptions of Lemma \ref{lem:1}.
Since $p_m^{-1}(w)<p_n^{-1}(w)$ for $w\in (0,\frac1m)$ and
$$\gamma(w)=\frac{h(p_m^{-1}(w))}{h(p_n^{-1}(w))},$$
it follows from Lemma  \ref{lem:1} that $\liminf_{w\to 0^+}\gamma(w)\leq 1$.\qed

\end{document}